\documentclass[12pt,a4paper]{amsart}
\usepackage{mathrsfs}
\usepackage{amsfonts}

\usepackage[active]{srcltx}
\usepackage{enumerate}

\usepackage{amsmath,amssymb,xspace,amsthm}

\newtheorem{theorem}{Theorem}
\newtheorem{remark}[theorem]{Remark}
\newtheorem{lemma}[theorem]{Lemma}
\newtheorem{proposition}[theorem]{Proposition}
\newtheorem{corollary}[theorem]{Corollary}
\newtheorem{conjecture}[theorem]{Conjecture}
\numberwithin{equation}{section}

\newcommand{\C}{\mathbb C}

\newcommand{\N}{\mathbb{N}}
\newcommand{\Z}{\mathbb{Z}}

\def\mg{\mathfrak{g}}
\def\ma{\mathfrak{a}}
\def\mh{\mathfrak{h}}
\def\s{\mathfrak{s}}

\def\sl{\mathfrak{sl}}
\def\gl{\mathfrak{gl}}

\newcommand{\cd}{\mathcal{D}}

\title[On the universal enveloping algebra of $\mathfrak{gl}_{n}$]{\bf An algebra isomorphism on
$U(\mathfrak{gl}_n)$}
\author{Yang Li and  Genqiang Liu}

\date{\today}

\begin{document}

\begin{abstract} For each  positive integer $n$,  let  $\mathfrak{s}_n=\mathfrak{gl}_n\ltimes \mathbb{C}^n$.
We show that $U(\mathfrak{s}_{n})_{X_{n}}\cong \mathcal{D}_{n}\otimes U(\mathfrak{s}_{n-1})$ for any  $n\in\mathbb{Z}_{\geq 2}$, where $U(\mathfrak{s}_{n})_{X_{n}}$ is the
localization of  $U(\mathfrak{s}_{n})$ with respect to the subset $X_n:=\{e_1^{i_1}\cdots e_{n}^{i_{n}}\mid i_1,\dots,i_{n}\in \mathbb{Z}_+\}$, and $\mathcal{D}_{n}$ is the Weyl algebra
$\mathbb{C}[x_1^{\pm 1}, \cdots, x_{n}^{\pm 1}, \frac{\partial}{\partial x_1},\cdots, \frac{\partial}{\partial x_{n}}]$.
As an application, we give a new proof of the Gelfand-Kirillov conjecture for $\mathfrak{s}_n$ and $\mathfrak{gl}_n$. Moreover we show that  the category of Harish-Chandra $U(\mathfrak{s}_{n})_{X_n}$-modules with a fixed weight support is
equivalent to the category  of finite dimensional $\mathfrak{s}_{n-1}$-modules whose representation type is wild,  for any $n\in \mathbb{Z}_{\geq 2}$.
\end{abstract}
\vspace{5mm}
\maketitle

\vspace{2mm}

\noindent{{\bf Math. Subj. Class.} 2010: 17B05, 17B10, 17B30, 17B35}

\section{introduction}
Weyl type algebras or  differential operator algebras  play an important role in the structure theory and representation theory of Lie algebras.
The famous  Gelfand-Kirillov conjecture (see \cite{GK}) states that the skew fraction field of the universal enveloping algebra of an algebraic Lie algebra over an algebraically closed field $\mathbb{K}$ is isomorphic to the skew field of fractions of some Weyl algebra over a purely transcendental extension of $\mathbb{K}$. The conjecture was settled by Gelfand and Kirillov \cite{GK} for nilpotent Lie algebras, and for $\gl_n$ and $\sl_n$. For recent results on
Gelfand-Kirillov  conjecture, we can see \cite{P}. In the  representation theory, an important application of Weyl algebras is the  Beilinson-Bernstein localization Theorem, see \cite{BB}.

Many important  finite dimensional Lie algebras in  mathematical physics are not semi-simple. Unlike complex semi-simple Lie algebras, the  theory of non semi-simple Lie algebras is still not well developed. In the last decade, there  were several important progresses on the representation of non semi-simple Lie algebras, see
\cite{BL, Du, DLMZ,LMZ, MM} and references therein.

For each  positive integer $n$, let  $\mathfrak{s}_n=\mathfrak{gl}_n\ltimes \mathbb{C}^n$ which is isomorphic to the maximal parabolic subalgebra
$\mathfrak{p}_n=\{(a_{ij})\in \sl_{n+1}\mid a_{n+1,k}=0,   1\leq k\leq n\}$ of $\sl_{n+1}$. The algebra $\mathfrak{p}_n$ is important for constructing
cuspidal $\sl_{n+1}$-modules, see \cite{GS}. On the other hand, the skew fraction field  of $U(\gl_{n+1})$ is generated by the skew fraction field of $U(\s_n)$  and  its $n+1$ central elements. These facts motivate us to study
$U(\mathfrak{s}_n)$ and its representations.

We use the Weyl algebra $\cd_n$  to study $\s_n$.  We show that there is an isomorphism between  $U(\mathfrak{s}_{n})_{X_{n}}$ and $\mathcal{D}_{n}\otimes U(\mathfrak{s}_{n-1})$ for any  $n\in\Z_{\geq 2}$, see Theorem \ref{tensor-iso}.
This isomorphism makes it possible to study $\s_n$ by induction on $n$.
As an application, we give a new proof of the Gelfand-Kirillov conjecture for $\mathfrak{s}_n$ and $\mathfrak{gl}_n$, see Theorem \ref{h-th} and Theorem
\ref{gk-th}. We can also reduce some infinite dimensional $\s_n$-modules to finite dimensional $\s_{n-1}$-modules, see Proposition \ref{mod-p}.

In this paper, we denote by $\Z$, $\N$, $\Z_+$ and $\C$ the sets of integers, positive integers, nonnegative
integers and complex numbers, respectively. All vector spaces and algebras are over $\C$. For a Lie algebra
$\mathfrak{g}$ we denote by $U(\mathfrak{g})$ its universal enveloping algebra. We write $\otimes$ for
$\otimes_{\mathbb{C}}$.

\section{Preliminaries}


 For each  $n\in \N$, let $\gl_{n}$ be the general linear Lie algebra over $\C$,  i.e., $\gl_{n}$  is consisting of $n\times n$-complex matrices.
 Let $e_{ij}$   denote the $n\times n $-matrix  unit whose $(i,j)$-entry is $1$ and $0$ elsewhere, $1\leq i,j \leq n$. Then $\{e_{ij}\mid 1\leq i,j \leq n\}$ is a basis of $\gl_{n}$.

  Let $A_{n}= \C [x_1^{\pm 1}, \dots, x_{n}^{\pm 1}]$ be the Laurent polynomial algebra in $n$ variables.
Then the subalgebra of $\text{End}_{\C}(A_{n})$  generated by $\{x_i,x_i^{-1},\frac{\partial}{\partial x_i}\mid 1\leq i\leq n\}  $ is called the Weyl algebra $\cd_{n}$ over $A_{n}$.
It is clear that $\cd_n\cong \cd_1\otimes \cdots \otimes \cd_1$. It is well known that there is an algebra homomorphism $\phi$ from $U(\gl_n)$ to $\cd_n$ such that $\phi(e_{ij})=x_i\frac{\partial}{\partial x_j}$, $1\leq i,j \leq n$.

Let  $\mathfrak{s}_n=\mathfrak{gl}_n\ltimes \mathbb{C}^n$ be the semidirect product of the general linear Lie algebra $\mathfrak{gl}_n$ and its natural representation $\mathbb{C}^n$. When $n=1$, $\s_1$ is the two dimensional Lie algebra $\C d_0\oplus \C d_1$ such that $[d_0,d_1]=d_1$. Let $\{e_1,\cdots,e_{n}\}$ be the standard basis of $\mathbb{C}^{n}$.

Let  $\mathfrak{h}_{n}=\oplus_{i=1}^{n}\C e_{ii} $ which is a Cartan subalgebra of $\s_n$.   An $\s_n$-module  $M$ is called a {\it weight module} if $\mathfrak{h}_{n}$ acts diagonally on  $M$, i.e.
$$ M=\oplus_{\lambda\in \mathfrak{h}_{n}^*} M_\lambda,$$
where $M_\lambda:=\{v\in M \mid hv=\lambda(h) v, \ \forall \ h\in \mh_n\}.$   Denote $$\mathrm{Supp}(M):=\{\lambda\in \mathfrak{h}_{n}^* \mid M_\lambda\neq0\}.$$
For a weight $\lambda\in \mathfrak{h}_{n}^*$, we  identify $\lambda$ with the
$n$-tuple $\lambda=(\lambda_1,\dots,\lambda_n)$, where $\lambda_i=\lambda(e_{ii})$. Conversely any vector $\mu=(\mu_1,\dots,\mu_n)\in \Z^n$ can be viewed as a
weight such that $\mu(e_{ii})=\mu_i$, for any $i$.
A nonzero element $v\in M_\lambda$ is called a weight vector.
A weight $\s_n$-module $M$ is called a Harish-Chandra module if all its weight spaces are finite dimensional.

\section{The tensor product decomposition of $U(\mathfrak{s}_{n})_{X_{n}}$}

In this section, we will study the structure of $U(\s_n)$ using the Weyl algebra $\cd_n$ and the localization technique.

\subsection{The  isomorphism theorem}

For every $n\in \Z_{\geq 2}$, let $\ma_{n-1}$ be the subalgebra of $\gl_{n}$ spanned by
$$e_{ij}-e_{jj}, \ \ 1\leq i,j\leq n.$$

\begin{lemma}\label{s-lemma}For every $n\in \Z_{\geq 2}$, the  linear map $ \xi: \s_{n-1}\rightarrow  \ma_{n-1}$ such that
\begin{equation}\aligned &e_{ij} \mapsto  \tilde{e}_{ij}:=e_{ij}-e_{nj},\\
e_i &\mapsto \tilde{e}_i:=\sum_{l=1}^n(e_{il}-e_{nl}) ,  \ 1\leq i,j\leq n-1,
\endaligned\end{equation}  is a Lie algebra isomorphism.
\end{lemma}
\begin{proof}   Since $\dim \s_{n-1}=\dim \ma_{n-1}$ and $\xi$ is surjective,  $\xi$ is bijective.
For any $ i,j,k,l\in \{1,\dots, n-1\}$, we have
$$\aligned\ [\tilde{e}_{ij},\tilde{e}_{kl}]&= [e_{ij}-e_{nj},e_{kl}-e_{nl}]\\
&=\delta_{jk}(e_{il}-e_{n,l})-\delta_{li}(e_{kj}-e_{nj})\\
&= \delta_{jk}\tilde{e}_{il}-\delta_{li}\tilde{e}_{kj},
\endaligned$$
$$\aligned\ [\tilde{e}_i,\tilde{e}_j]&=
 [\sum_{l=1}^n(e_{il}-e_{nl}),\sum_{k=1}^n(e_{jk}-e_{nk})]\\
&= 0,
\endaligned$$
and
$$\aligned\ [\tilde{e}_{ij},\tilde{e}_{k}]&= [e_{ij}-e_{nj},\sum_{l=1}^n(e_{kl}-e_{nl})]\\
&=\delta_{jk}\sum_{l=1}^n(e_{il}-e_{nl})\\
&= \delta_{jk}\tilde{e}_{i}.
\endaligned$$
Therefore $\xi$ is a Lie algebra isomorphism.
\end{proof}
 
 In $U(\s_{n})$, every $e_k$ is a locally $\text{ad}$-nilpotent element.  So  the subset $X_n:=\{e_1^{i_1}\cdots e_{n}^{i_{n}}\mid i_1,\dots,i_{n}\in \mathbb{Z}_+\}$ is an ore subset of $U(\s_{n})$, see Lemma 4.2  in \cite{M}.  We use $U(\mathfrak{s}_{n})_{X_{n}}$  to denote the localization of  $U(\mathfrak{s}_{n})$ with respect to the subset $X_n$.

\begin{proposition}\label{s-iso} For any  $n\in \Z_{\geq 2}$, the linear map $\psi$ defined by
\begin{equation}\aligned U(\mathfrak{s}_{n})_{X_{n}} & \rightarrow \cd_{n}\otimes U(\ma_{n-1}),\\
e_{ij}& \mapsto x_i\frac{\partial}{\partial x_j}\otimes 1 + x_ix_j^{-1}\otimes (e_{ij}-e_{jj}),\\
e_i&\mapsto x_i\otimes 1,\\
e_i^{-1}&\mapsto x^{-1}_i\otimes 1,\ \  1\leq i\leq n,
\endaligned \end{equation}  is an associative algebra isomorphism.
\end{proposition}
\begin{proof} For any $u\in \cd_{n}$ and $v\in U(\ma_{n-1})$,  we also use $u$ and  $v$ to denote the elements $u\otimes 1$ and $1\otimes v$  in $\cd_{n}\otimes U(\ma_{n-1})$ respectively. From
 $$\aligned &\psi([e_{ij},e_{kl}])=\psi(\delta_{jk}e_{il}-\delta_{li}e_{kj})\\
 &=\delta_{jk}x_i\frac{\partial}{\partial x_l}\otimes 1 +\delta_{jk} x_ix_l^{-1}\otimes (e_{il}-e_{ll})\\
 &\ \ \ \  -\delta_{li}x_k\frac{\partial}{\partial x_j}\otimes 1 -\delta_{li} x_kx_j^{-1}\otimes (e_{kj}-e_{jj})\\
 &= [\psi(e_{ij}),\psi(e_{kl})]
 \endaligned$$
 and $\psi([e_{ij},e_k])=[\psi(e_{ij}),\psi(e_k)]$,
 we can see the $\psi$ preserves the Lie bracket relations of $\s_{n}$. So $\psi$ is an algebra homomorphism.
 It can be also checked directly that the  algebra homomorphism
$$\aligned \psi': \cd_{n}\otimes U(\ma_{n-1}) &
 \rightarrow U(\mathfrak{s}_{n})_{X_{n}} ,\\
e_{ij}-e_{jj} & \mapsto e_je_i^{-1}e_{ij}-e_{jj},\\
x_i&\mapsto e_i,\\
x_i^{-1}&\mapsto e^{-1}_i,\ \  1\leq i\leq n,
\endaligned $$
is the inverse of $\psi$.
 Therefore $\psi$ is an isomorphism.
\end{proof}

Combining Lemma \ref{s-lemma} with Proposition \ref{s-iso}, we have the following tensor product decomposition of $U(\mathfrak{s}_{n})_{X_{n}}$.

\begin{theorem}\label{tensor-iso} We have the associative algebra isomorphism $$ U(\mathfrak{s}_{n})_{X_{n}}\cong \cd_{n}\otimes U(\s_{n-1}),$$ for any $n\in \Z_{\geq 2}$.
\end{theorem}

\begin{remark} This isomorphism in Theorem \ref{tensor-iso} is motivated by the tensor field modules over the Witt algebra, see \cite{Sh}.
\end{remark}
\subsection{The Gelfand-Kirillov conjecture for $\mathfrak{s}_n$ and $\mathfrak{gl}_n$}

For any $r, s\in \Z_+$, define $\cd_{r,s}=\cd_r\otimes \mathbb{C}[y_1,\dots,y_s]$ which is a Noetherian domain. Denote its skew field of fractions by $\mathbf{F}_{r,s}$.
 If $s=0$, then we denote $\mathbf{F}_{r,0}$ by $\mathbf{F}_{r}$.

 In 1966, in the  famous paper \cite{GK},  Gelfand and Kirillov put forward the following conjecture.

\begin{conjecture}If  $\mg$ is the Lie algebra of a linear algebraic group over $\mathbb{C}$, then the  skew fraction field $\mathbf{F}(\mg)$ of $U(\mg)$ is isomorphic to $\mathbf{F}_{r,s}$ for some
$r,s\in \Z_+$ depending $\mg$.
\end{conjecture}
The conjecture was settled for nilpotent Lie algebras, $\sl_n$ and $\gl_n$
by Gelfand and Kirillov in the same paper  \cite{GK}.

In 1979, Nghi\^em Xu\^an Hai showed that the Gelfand-Kirillov conjecture is true for the semi-direct products $\sl(n), \mathfrak{sp}(2n)$ and $\mathfrak{so}(n)$ with their standard representation. Next using Theorem \ref{tensor-iso}, we give another proof for the Gelfand-Kirillov conjecture of $\s_n$.

\begin{theorem}\label{h-th}The skew fraction field $\mathbf{F}(\s_n)$ of $U(\s_n)$ is isomorphic to $\mathbf{F}_{\frac{n(n+1)}{2}}$,  for any $n\in \Z_{\geq 1}$.
\end{theorem}
\begin{proof}  We proceed by induction on $n$. If $n=1$, then $\s_1$ is the two dimensional non-abelian Lie algebra. Recall that $\s_1=\C d_0\oplus \C d_1$ such that $[d_0,d_1]=d_1$. Let
$U(\s_1)_{d_1}$ denote the localization of $U(\s_1)$ respect to the subset $\{d_1^i\mid i\in \Z_+\}$. Then the map $$U(\s_1)_{d_1}\rightarrow \cd_1: d_1^{-1}d_0\mapsto \frac{\partial }{\partial x_1}, d_1\mapsto x_1$$ can give an isomorphism between $\mathbf{F}(\s_1)$ and the
 the  skew fraction field of $\cd_1$. For arbitrary positive integer $n$, by the isomorphism in Theorem \ref{tensor-iso},
$\mathbf{F}(\s_n)\cong \mathbf{F}_n\otimes \mathbf{F}(\s_{n-1})$. Then by the inductive hypothesis, $\mathbf{F}(\s_{n-1})\cong \mathbf{F}_{\frac{n(n-1)}{2}}$.
Consequently $\mathbf{F}(\s_n)\cong \mathbf{F}_{\frac{n(n+1)}{2}}$.
\end{proof}

\begin{corollary} For any $n\in \Z_{\geq 1}$, the center of $U(\s_n)$ is $\C$.
\end{corollary}
\begin{proof} By Proposition 5.12 in Chapter $1$ of \cite{D}, the center of $\cd_n$ is $\C$. Then the center of the Weyl field $\mathbf{F}_n$ is also $\C$. Hence so is the center of $U(\s_n)$ by Theorem \ref{h-th}.
\end{proof}

Using Theorem \ref{h-th} and the generators of the center of $U(\gl_n)$,
we give a proof of the Gelfand-Kirillov conjecture for $\gl_n$.

\begin{theorem}[Gelfand and Kirillov, 1966]\label{gk-th} For any $n\in\Z_{\geq 2}$, the skew fraction field $\mathbf{F}(\gl_n)$ of $U(\gl_n)$ is isomorphic to $\mathbf{F}_{\frac{n(n-1)}{2},n}$.
\end{theorem}
\begin{proof} For every $k\in \{1,\dots,n\}$, let
$$c_{k } \ = \ \displaystyle {\sum_{(i_1,\ldots,i_k)\in \{
1,\ldots,n \}^k}} E_{i_1 i_2}E_{i_2 i_3}\ldots E_{i_k i_1}.$$
It is well known that the center $Z_n$ of $U(\gl_n)$
is generated by $c_{1},\dots,c_{n}$, see \cite{GT}.
By the result in \cite{Di}, $U(\gl_n)\subset \mathbf{F}(\mathfrak{p}_{n-1})Z_n$, where $\mathfrak{p}_{n-1}$ is the subalgebra of $\gl_n$ spanned by $e_{i,j},1\leq i\leq n-1, 1\leq j\leq n$. Note that $\mathfrak{p}_{n-1}\cong \s_{n-1}$.
Then by  Theorem \ref{h-th}, we have that $\mathbf{F}(\gl_n)\cong \mathbf{F}_{\frac{n(n-1)}{2},n}$.
\end{proof}
\begin{remark} Although Gelfand and Kirillov used the induction method in \cite{GK},
but they did not give the isomorphism in Theorem \ref{tensor-iso}.
In 2010, Premet showed that the Gelfand-Kirillov conjecture  does not hold for
simple Lie algebras of types $B_n, n \geq  3, D_n, n \geq  4, E_6, E_7, E_8$ and $ F_4$, see \cite{P}. The Gelfand-Kirillov conjecture remains open for simple Lie algebras of types $C_n, G_2$.
We expect that the tensor product decomposition method for $U(\s_n)$ may be useful to the study of $U(\mathfrak{sp}_{2n})$.
\end{remark}

\subsection{$\s_n$-modules}
Using the isomorphism in Theorem \ref{tensor-iso}, we can study weight  $\s_n$-modules.
For any $\lambda=(\lambda_1,\dots,\lambda_n)\in \C^n$, and  any $\s_{n-1}$-module $V$, the tensor product $x^{\lambda}A_n\otimes V$ can be defined to be an $\s_n$-module denoted by $T(\lambda,V)$ via the  isomorphism in Theorem \ref{tensor-iso}, where $x^{\lambda}=x_1^{\lambda_1}\cdots x_n^{\lambda_n}$.

%

\begin{lemma}\label{epic-lem}If $M=\oplus_{\alpha\in \Z^n} M_{\lambda+\alpha}$ is a  Harish-Chandra $\mathfrak{s}_{n}$-modules with each $e_i, 1\leq i\leq n$ acting on it bijectively, then $M\cong T(\lambda, V)$ for some $\lambda\in \C^n$, and some  $\s_{n-1}$-module $V$.
\end{lemma}
\begin{proof} Using the isomorphism in Theorem \ref{tensor-iso}, we view  $M$  as
a  module over  $\cd_n\otimes U(\s_{n-1})$. Moreover $M$ is a weight $\cd_n$-module with respect to the action of $\{ x_1\frac{\partial }{\partial x_1},\dots,x_n\frac{\partial }{\partial x_n}\}$. Set $V=M_\lambda$ which is a $U(\s_{n-1})$-modules, since $[\cd_n, U(\s_{n-1})]=0$.  As a vector space $$M=\oplus_{\alpha\in \Z^n} x^\alpha\otimes V= x^\lambda A_n \otimes  V.$$
For any $i\in\{1,\dots, n\}$, we have $$ \frac{\partial }{\partial x_i}\cdot x^\alpha\otimes v=x_i^{-1}\cdot (x_i \frac{\partial }{\partial x_i})\cdot x^\alpha\otimes v=(\lambda_i+\alpha_i) x^{\alpha-\epsilon_i}\otimes v.$$ So as a  $\cd_n\otimes U(\s_{n-1})$-module,
$M\cong  x^\lambda A_n \otimes  V$. Consequently,  $M \cong T(\lambda, V)$ as an $\s_n$-module.
\end{proof}

\begin{proposition}\label{mod-p} For a fixed $\lambda\in \C^n$, let $\mathcal{C}_{\lambda,n}$
be the category of Harish-Chandra $U(\mathfrak{s}_{n})_{X_n}$-modules $M$ such that $\text{Supp}(M)=\lambda+\Z^n$. Then $\mathcal{C}_{\lambda,n}$ is
equivalent to the category $\mathcal{A}_{n-1}$ of finite dimensional $\s_{n-1}$-modules, where $n\in \Z_{\geq 2}$.
\end{proposition}
\begin{proof}We will show that the functor
$$\aligned
T(\lambda, -):\   \mathcal{A}_{n-1}&\rightarrow \mathcal{C}_{\lambda,n},\ \
V&\mapsto T(\lambda,V),
\endaligned$$
 is an equivalence of the two categories.
By Schur's Lemma,  we have $\text{End}_{\cd_n}(x^\lambda A_n)=\C$.
So for any $V,W\in \mathcal{A}_{n-1}$, $$ \text{Hom}_{\cd_n\otimes U(\s_{n-1})}(T(\lambda,V),T(\lambda,W))
\cong \text{Hom}_{\s_{n-1}}(V,W).$$ So the functor  $T(\lambda, -)$ is fully faithful. By Lemma \ref{epic-lem},  $T(\lambda, -)$ is an equivalence.
 \end{proof}

Let $\C\langle x, y\rangle$ be the free associative algebra over $\C$ generated by two variables $x, y$.  Recall that an abelian category $\mathcal{C}$ is wild if there exists an exact functor from the category of finite dimensional $\C\langle x, y\rangle$-modules to $\mathcal{C}$, preserving indecomposability and mapping non-isomorphic modules to
 non-isomorphic objects.

\begin{corollary}For any $n\in\Z_{\geq 2}$, $\lambda\in \C^n$, the category $\mathcal{C}_{\lambda,n}$ is wild.
\end{corollary}
\begin{proof} By the results in \cite{Mak}, the category $\mathcal{A}_{n-1}$ is wild. So does $\mathcal{C}_{\lambda,n}$.
\end{proof}

\begin{remark} In \cite{Di}, it was shown that there is an element $d\in U(\gl_n)$
such that the localization of $U(\gl_n)_d$ is isomorphic to
$U(\s_{n-1})_d\otimes Z(U(\gl_n))$. So the study of $\s_{n-1}$-modules may be
meaningful for constructing $U(\gl_n)$-modules.
\end{remark}

\vspace{2mm}
\noindent
{\bf Acknowledgments. }This research is supported  by NSF of China (Grant
11771122) and NSF of Henan Province (Grant 202300410046).

\vspace{0.2cm}

\noindent Y. Li: School of Mathematics and Statistics, Henan University, Kaifeng
475004, China. Email: 897981524@qq.com

\vspace{0.2cm}
\noindent G. Liu: School
of Mathematics and Statistics, Henan University, Kaifeng 475004, P.R. China. Email:
liugenqiang@henu.edu.cn
\vspace{0.2cm}


\begin{thebibliography}{99999}
\bibitem[BB]{BB}  A. Beilinson, J. N. Bernstein, Localisation de g-modules, C. R. Acad. Sci. Ser. A-B 292 (1981), 15-18.

\bibitem[BL]{BL} V. V. Bavula, T. Lu, The universal enveloping algebra of the Schr{\"o}dinger algebra and its prime spectrum, Canad. Math. Bull. 61 (2018), no. 4, 688-703.


\bibitem[D]{D} D. Mili\u{c}i\'{c}, Lectures on algebraic theory of $D$-modules,   \text{http://www.math.utah.edu/~milicic/Eprints/dmodules.pdf.}

\bibitem[Du]{Du}B. Dubsky, Classification of simple weight modules with finite-dimensional
weight spaces over the Schrodinger algebra, Lin. Algebra Appl.  {\bf 443} (2014),
204--214 .

    \bibitem[DLMZ]{DLMZ} B. Dubsky, R. Lu, V. Mazorchuk, K.  Zhao,
Category $\mathcal{O} $ for the Schr{\"o}dinger algebra, Linear Algebra Appl.   460 (2014), 17-50.

\bibitem[Di]{Di} J. Dixmier, Sur les alg\'ebres enveloppantes de $\sl(n, \C$) et $\text{af}(n, \C)$, Bull. sci. math. 1(1976), 57-95.

\bibitem[GK]{GK}    I.M. Gelfand, A.A. Kirillov, Sur les corps li\'es aux alg\'ebres enveloppantes des alg\'ebres de Lie, Publ. Math. Inst. Hautes \'Etudes Sci. 31 (1966) 5-19.

\bibitem[GS]{GS}    D. Grantcharov, V. Serganova, Cuspidal representations of $\sl(n+1)$, Adv. Math. 224 (2010), 1517-1547.

\bibitem[GT]{GT} I. Gelfand, M. Tsetlin, Finite-dimensional representations of the group of unimodular matrices, {Doklady Akad. Nauk SSSR (N.s.)}, {71} (1950), 825-828.

\bibitem[H]{H} N. X. Hai, R\'eduction de produit semi-directs et conjecture de Gelfand et Kirillov, Bull. Soc. Math. France (1979) 241-267.

    \bibitem[L]{L} H. Li, On certain categories of modules for affine Lie algebras, Math. Z. 248 (3) (2004) 635-664.

\bibitem[LMZ]{LMZ}  R. Lu, V. Mazorchuk and K. Zhao, On simple modules over conformal Galilei algebras, J. Pure Appl. Algebra 218 (2014), 1885-1899.

\bibitem[Mak]{Mak} E. Makedonskii,  On wild and tame finite dimensional Lie algebras, Functional analysis
and its applications 47(4) (2013), 271-283.

\bibitem[M]{M} O. Mathieu, Classification of irreducible weight modules, Ann. Inst. Fourier. {\bf 50} (2000),  537-592.

\bibitem[MM]{MM} V. Mazorchuk, R. Mrden, Lie algebra modules which are locally finite over the semi-simple part, (2020),
arXiv:2001.02967 [math.RT], to appear in Nagoya Math. J.

\bibitem[Sh]{Sh} G. Shen, Graded modules of graded Lie algebras of
Cartan type. I. Mixed products of modules,  Sci. Sinica Ser. A {\bf
29}  (1986),  no. 6, 570-581.



 \bibitem[P]{P}A. Premet, Modular Lie algebras and the Gelfand-Kirillov conjecture, Invent. Math. 181 (2010), no. 2, 395-420.

\end{thebibliography}
\end{document}